\documentclass[11pt]{article}
\usepackage[margin=1in]{geometry}
\usepackage{amsmath,amssymb,amsthm,mathtools}
\usepackage{booktabs}

\usepackage{tikz}
\usepackage{hyperref}
\hypersetup{hidelinks}
\usepackage{microtype}
\usepackage{enumitem}
\theoremstyle{plain}
\newtheorem{theorem}{Theorem}[section]
\newtheorem{lemma}[theorem]{Lemma}

\theoremstyle{definition}

\theoremstyle{remark}

\newcommand{\AT}{\operatorname{AT}}

\usepackage{xcolor}
\begin{document}
\title{The Alon-Tarsi number of 4-regular circulant graphs}
\author{Xinbo Xu$^1$ \and Xuding Zhu$^1$\thanks{ Grant number:  NSFC 12371359.} }
\date{%
    $^1$School of Mathematical Sciences, Zhejiang Normal University\\[2ex]%
    \today
}
\maketitle
\begin{abstract}
This paper studies the Alon-Tarsi number of connected $4$-regular circulant graphs 
$C_n(s,t)$, where $1\le s<t<n/2$.
We prove that if $n=\ell p^r$, where $p$ is an odd prime, $r \ge 1$,
and $t\equiv\pm s\pmod{\ell}$, then 
 $\operatorname{AT}(C_n(s,t))=3$
if $\ell$ is even, or if $\ell \ge 3$ is odd and
$p\nmid 2^\ell-2$.
  We also prove that $\AT(C_{6m}(2,3))=3$ for  $m\ge2$, and
$\operatorname{AT}(C_n(1,3))=3$ for   $n\ge8$.

\end{abstract}
\section{Introduction}

 The Alon-Tarsi number   is a graph invariant defined through graph polynomial.
  For a graph $G$ with a fixed orientation, the graph polynomial $P_G(\mathbf x)$ is defined as 
 \begin{equation}\label{eq:gpoly}
P_G(\mathbf x)=\prod_{(u,v) \in E(G)}(x_u-x_v).
\end{equation}
For $  g\in {\mathbb Z}_{\ge 0}^{V(G)}$, let $\left[\prod_{v\in V(G)}x_v^{g(v)}\right]P_G(\mathbf x)$ be the coefficient of the monomial 
\[
 \prod_{v\in V(G)}x_v^{g(v)}
\]
when $P_G$ is written as a linear combination of monomials. If $  g(v)=k$ for every $v$, we write $\left[\prod_{v\in V(G)}x_v^{g(v)}\right]P_G(\mathbf x)$ as $\left[\prod_{v\in V(G)}x_v^{k}\right]P_G(\mathbf x)$.  We say the monomial $\prod_{v\in V(G)}x_v^{g(v)} $ is \emph{non-vanishing} in $P_G$ if $\left[\prod_{v\in V(G)}x_v^{g(v)}\right]P_G(\mathbf x) \ne 0$. 
The \emph{ Alon-Tarsi number} $\operatorname{AT}(G)$ of $G$ is the minimum integer $k$ such that $P_G$ has a non-vanishing monomial $\prod_{v\in V(G)}x_v^{g(v)} $ with $g(v) < k$ for all $v$. 

It follows from  the Combinatorial Nullstellensatz that $\operatorname{AT}(G)$ is an upper bound for the choice number $ch(G)$ of $G$ \cite{AlonTarsi1992}, as well as the on-line choice number $\chi_P(G)$ of $G$ \cite{Schauz2010}. 
 
 This paper studies the Alon-Tarsi number of  $4$-regular circulant graphs $C_n(s,t)$, which has   vertex set $\mathbb Z_n$, and in which each $i\in\mathbb Z_n$ is adjacent to $i\pm s$ and $i\pm t$ (modulo $n$). Without loss of generality, we may assume that $1 \le s < t < n/2$. We may assume that $C_n(s,t)$ is connected, which is equivalent to $\gcd(n,s,t)=1$. 
 
 The Alon-Tarsi numbers of $4$-regular graphs have been studied in many papers. For a $4$-regular graph $G$,  ${\rm deg}(P_G) = 2|V(G)|$. It follows that $\operatorname{AT}(G) \ge 3$. 
The Alon--Tarsi version of Brooks' theorem
\cite{HladkyKralSchauz2010} implies that every connected
$4$-regular graph $G$ other than $K_5$ satisfies
$\operatorname{AT}(G)\le4$. A natural problem is which $4$-regular graphs $G$ have $\operatorname{AT}(G) =3$.
 Fleischner and Stiebitz  \cite{FleischnerStiebitz1992} proved that if a 4-regular graph $G$ decomposes into a Hamilton cycle and vertex-disjoint triangles, then  $\operatorname{AT}(G)=3$. This confirms a long standing conjecture of Erd\H{o}s on the chromatic number of such graphs. 
Tarsi~\cite{Tarsi1999} proved that if a 4-regular graph $G$ has an odd number of $3$-colourings (up to permutations of colour names), then 
$\operatorname{AT}(G)=3$.  In \cite{AT-ToroidalGrids2023}, Li,  Shao,  Petrov  and Gordeev determined the Alon-Tarsi numbers of the Cartesian products $C_n \Box C_m$, which are 4-regular Cayley graphs:   $AT(C_n \Box C_m)=4$ if both $n,m$ are odd, and $AT(C_n \Box C_m)=3$ otherwise. In \cite[Problem~2.1]{CyclesColourings2025},     Zden\v{e}k Dvo\v{r}\'{a}k asked whether a random $4$-regular graph almost surely has Alon-Tarsi number 3, and the question remains open.

The Alon--Tarsi number of $4$-regular circulant graphs are also studied in the literature. 
Li, Ye and Shao~\cite{LiYeShao2023} determined the Alon--Tarsi number of a family of $4$-regular planar graphs
$R_m$. The graph  $R_m$ defined in \cite{LiYeShao2023} is isomorphic to $C_{2m}(1,2)$. They proved that for $m \ge 3$,
\[
\operatorname{AT}(C_{2m}(1,2))
=
\begin{cases}
3,&3\mid m,\\
4,&3\nmid m.
\end{cases}
\]

We remark that  colouring and fractional colouring of circulant graphs  are related to the  classical Motzkin problem  \cite{CLZ1999}, and the famous``lonely runner" conjecture \cite{Lonelyrunner2025}, and has been studied extensively in the literature. For $4$-regular circulant graphs, Heuberger \cite{Heubergerchromaticofcirculant} determined their chromatic numbers. Their choice numbers and fractional chromatic numbers remain undetermined in general, although for many special families, these values are determined.

Let $G=C_n(s,t)$, and let 
\begin{equation}\label{eq:cyclepoly}
P_G(\mathbf x)=\prod_{i=0}^{n-1}(x_i-x_{i+s})(x_i-x_{i+t}),
\end{equation}
where the calculation in the indices are modulo $n$.

The following are the main results of this paper:

\begin{theorem}\label{thm:prime}
If $n=\ell p^r$, where $p$ is an odd prime and $\ell \ge 3$ is odd, $r \ge 1$,
and $t\equiv\pm s\pmod{\ell}$, then
\begin{equation}\label{eq:congruence}
\left[\prod_{i=0}^{n-1}x_i^2\right]P_{C_n(s,t)}(\mathbf x)
\equiv
\begin{cases}
2^\ell-2,&t\equiv-s\pmod\ell,\\
2-2^\ell,&t\equiv s\pmod\ell
\end{cases}
\quad(\bmod p).
\end{equation}
If $p\nmid 2^\ell-2$, then $\AT(C_n(s,t))=3$.
\end{theorem}

\begin{theorem}\label{thm:c6m}
For every integer $m\ge2$, $\AT(C_{6m}(2,3))=3$.
\end{theorem}

\begin{theorem}\label{thm:c13}
For every integer $n\ge7$,  
\[
\left[\prod_{i=0}^{n-1}x_i^2\right]P_{C_n(1,3)}
=
N_3(C_n(1,3)),
\]
where for a graph $G$, $N_3(G)$ is the number of $3$-colourings of 
$G$ with labelled colours.
Consequently,
$\operatorname{AT}(C_n(1,3))=3$ for every $n\ge8$.
\end{theorem}

\section{Proofs of the main results}\label{sec:circulants}

\subsection{The proof of Theorem~\ref{thm:prime}}\label{subsec:prime}

Let $\ell,p,r,n,s,t$ satisfy the hypotheses of Theorem~\ref{thm:prime},
and put $q=p^r$. Note that $\gcd(n,s,t)=1$ and $t \equiv \pm s \pmod{\ell}$ imply that 
$\gcd(s, \ell)=1$. To expand $P_G(\mathbf x)$ as a linear combination of monomials, we select one of $x_u$ and $-x_v$ from each term $(x_u-x_v)$ and multiply them to form a monomial. We record such a selection by a   matrix $B=(b_{i,d})\in\{0,1\}^{n\times2}$, 
whose rows are indexed by
$i\in\mathbb{Z}_n$ and   columns are indexed by  $s,t$. For $i \in \mathbb{Z}_n$ and $d \in \{s,t\}$,   $b_{i,d}=0$ (respectively, $b_{i,d}=1$) means we select $x_i$ (respectively, $-x_{i+d}$) from the term $(x_i-x_{i+d})$. 

Given such a matrix $B$, put
\[
h(B)=\sum_{i\in\mathbb{Z}_n}\sum_{d\in\{s,t\}}b_{i,d}
\]
and
\[
t_i^B
=
2+\sum_{d\in\{s,t\}}(b_{i-d,d}-b_{i,d})
\qquad(i\in\mathbb{Z}_n).
\]
The selection recorded by $B$ contributes 
$(-1)^{h(B)} \prod_{i=0}^{n-1} x_i^{t_i^B}$ to 
$P_{C_n(s,t)}(\mathbf x)$. Hence 
$$P_{C_n(s,t)}(\mathbf x) = \sum (-1)^{h(B)} \prod_{i=0}^{n-1} x_i^{t_i^B},$$
where the summation is over all matrices $B \in \{0,1\}^{n \times 2}$. 

We call $B$ \emph{admissible} if $t_i^B=2$ for $i\in\mathbb{Z}_n$. Note that $B$ is admissible if and only if
\begin{equation}\label{eq:balance}
b_{i,s}+b_{i,t}=b_{i-s,s}+b_{i-t,t}\qquad \forall i\in\mathbb{Z}_n.
\end{equation}

Let $\mathcal{B}$ be the family of admissible matrices $B$. Then 
$$\left[\prod_{i=0}^{n-1}x_i^2\right]P_{C_n(s,t)}(\mathbf x) = \sum_{B \in \mathcal{B}} (-1)^{h(B)}.$$

 Let $\tau$ be the action on $\mathcal{B}$   defined as  
$(\tau B)_{i,d}=b_{i-\ell,d}$, with row indices modulo $n$. Note that
Equation~\eqref{eq:balance} shows that $\tau B \in \mathcal{B}$   whenever
$B \in \mathcal{B}$, and $h(\tau B)=h(B)$. 

Since $\tau^{p^r}$ is the identity, each orbit of $\tau$ has size dividing $p^r$. Denote by $[B]$ the orbit of $\tau$ containing $B$. 
If $[B]$ is not a singleton orbit, then 
$p| |[B]|$, and  hence $\sum_{B' \in [B]} (-1)^{h(B')} \equiv 0 \pmod{p}$.  

Let $$\mathcal{B}' = \{B \in \mathcal{B}: | [B]| =1\}.$$
Then 
$$ \left[\prod_{i=0}^{n-1}x_i^2\right]P_{C_n(s,t)}(\mathbf x) \equiv \sum_{B \in \mathcal{B}'} (-1)^{h(B)} \pmod{p}.$$

 For each $B\in\mathcal B'$,
the condition $|[B]|=1$ (i.e., $\tau B=B$ ) implies that the rows
$j,j+\ell,\ldots,j+(q-1)\ell$ are identical.
We therefore compress $B$ to a matrix
$\tilde{B}$ formed by its first
$\ell$ rows, where \[
\widetilde B
=
(b_{j,d})_{\substack{0\le j<\ell\\ d\in\{s,t\}}}
.\]
Let $y_0,\ldots,y_{\ell-1}$ be variables, and put
\begin{equation}\label{eq:py}
P(y_0,\ldots,y_{\ell-1})
=
\prod_{j=0}^{\ell-1}(y_j-y_{j+s})(y_j-y_{j+t}),
\end{equation}
where calculations in indices are modulo $\ell$.

 Similarly to the discussion before, we have 

$$\left[\prod_{j=0}^{\ell-1}y_j^2\right]
P(y_0,\ldots,y_{\ell-1}) = \sum_{B \in\mathcal B'}(-1)^{h(\tilde{B})}$$
It follows from the definition of $\tilde{B}$ that  $h(B) = qh(\tilde{B})$. As $q$ is odd, 
$(-1)^{h(\tilde{B})} =(-1)^{h(B)}.$ 
Therefore  
$$\left[\prod_{i=0}^{n-1}x_i^2\right]P_{C_n(s,t)}(\mathbf x)\equiv\left[\prod_{j=0}^{\ell-1}y_j^2\right]
P(y_0,\ldots,y_{\ell-1}) = \sum_{B \in \mathcal{B}'} (-1)^{h(\tilde{B})}   \pmod{p}.$$

If $t \equiv -s \pmod{\ell}$, then $(y_{j-t}-y_j) = -(y_j - y_{j+s})$. By re-grouping the terms in ~\eqref{eq:py}, we have
\begin{equation}\label{eq:py2}
P(y_0,\ldots,y_{\ell-1})
=
(-1)^\ell \prod_{j=0}^{\ell-1}(y_j-y_{j+s})^2,
\end{equation}
Note that 
$$(y_j-y_{j+s})^2 = y_j^2- 2y_jy_{j+s} + y_{j+s}^2.$$
To expand $P(y_0,\ldots,y_{\ell-1})$, for each term $y_j^2- 2y_jy_{j+s} + y_{j+s}^2$, we need to choose one of the monomials to take part in the multiplication. 
For each $j$, let $a_j\in\{0,1,2\}$ be the exponent of
$y_j$ in the term selected from $y_j^2- 2y_jy_{j+s} + y_{j+s}^2$.
The total exponent of $y_j$ in the selected product is
$2-a_{j-s}+a_j$.
Since $(\ell, s) =1$, a selection contributes to $\prod_{j=0}^{\ell-1}y_j^2$
if and only if all the $a_j$ are equal.
The three constant choices contribute $1$, $(-2)^{\ell}$,
and $1$ to $\left[\prod_{j=0}^{\ell-1}y_j^2\right]
\prod_{j=0}^{\ell-1}(y_j-y_{j+s})^2$, respectively. Hence
\[
\left[\prod_{j=0}^{\ell-1}y_j^2\right]
\prod_{j=0}^{\ell-1}(y_j-y_{j+s})^2
=2+(-2)^{\ell},
\]
and 
\[
\left[\prod_{j=0}^{\ell-1}y_j^2\right]
P(y_0,\ldots,y_{\ell-1}) = 2^\ell-2.
\]

If $t \equiv s \pmod{\ell}$, then by the same calculation, we have
\[
\left[\prod_{j=0}^{\ell-1}y_j^2\right]
P(y_0,\ldots,y_{\ell-1}) = 2- 2^\ell.
\]
Hence, Equation (\ref{eq:congruence}) holds, and if  $p\nmid2^\ell-2$, then $\AT(C_n(s,t))=3$.
This completes the proof of Theorem \ref{thm:prime}.

Note that if $n$ is even and both $s,t$ are odd, then $C_n(s,t)$ is a bipartite graph. It is well-known \cite{ZhuBalakrishnan2022} that for any   bipartite graph $G$, $\AT(G) = \max\{ \lceil |E(H)|/|V(H)| \rceil +1: H \subseteq G\}$. Hence 4-regular bipartite graphs have Alon-Tarsi number $3$.  In particular, if $\ell$ is even and $\gcd (\ell, s) = \gcd(\ell, t)=1 $ and $n= \ell p^r$, then $\AT(C_n(s,t))=3$. 

\subsection{The proof of Theorem~\ref{thm:c6m}}
\label{subsec:c6m}

We compute the coefficient of $\prod_{i=0}^{n-1}x_i^2$
in $P_{C_n(2,3)}(\mathbf x)$ using the following well-known interpolation
formula~\cite{ZhuBalakrishnan2022}.

\begin{lemma}\label{lem:coefficient}
Let $\mathbb F$ be a field, and let
$f\in\mathbb F[x_0,\ldots,x_{n-1}]$ have degree at most
$\sum_{i=0}^{n-1}t_i$, where the $t_i$ are nonnegative integers.
For each $i$, let $S_i\subseteq\mathbb F$ have size $t_i+1$.
Put $\mathcal S=\prod_{i=0}^{n-1}S_i$ and, for
$a=(a_0,\ldots,a_{n-1})\in\mathcal S$, put
$$
N_{\mathcal S}(a)
=\prod_{i=0}^{n-1}\prod_{b\in S_i\setminus\{a_i\}}(a_i-b).
$$
Then
\begin{equation}\label{eq:interpolation}
[x_0^{t_0}\cdots x_{n-1}^{t_{n-1}}]f
=\sum_{a\in\mathcal S}\frac{f(a)}{N_{\mathcal S}(a)}.
\end{equation}
\end{lemma}

Let $n\ge7$ and $G=C_n(2,3)$.
We use the orientation in~\eqref{eq:cyclepoly}.
We apply Lemma~\ref{lem:coefficient} over $\mathbb Q$ with
$t_i=2$ and $S_i=S=\{0,1,2\}$ for every $i$.
For $d\in S$, define
\begin{equation}\label{eq:delta}
\delta(d)=\prod_{e\in S\setminus\{d\}}(d-e).
\end{equation}
Thus $\delta(0)=2$, $\delta(1)=-1$, and $\delta(2)=2$.
For a map $\varphi:V(G)\to S$, put
$D(\varphi)=\prod_{i=0}^{n-1}\delta(\varphi(i))$.
This is $N_{S^n}(a)$ with $a_i=\varphi(i)$.
Since $G$ is $4$-regular, $P_G$ has degree $2n$, and
Lemma~\ref{lem:coefficient} gives
\begin{equation}\label{eq:colouringsum}
\left[\prod_{i=0}^{n-1}x_i^2\right]P_G(\mathbf x)
=\sum_{\varphi:V(G)\to S}\frac{P_G(\varphi)}{D(\varphi)}
=\sum_{\varphi\text{ proper}}\frac{P_G(\varphi)}{D(\varphi)}.
\end{equation}
The second equality holds because $P_G(\varphi)=0$ if $\varphi$ is not a proper colouring.

In the  study of the Alon-Tarsi number of toroidal grids, Li, Shao, Petrov and Gordeev~\cite{AT-ToroidalGrids2023}
used the interpolation formula to express a coefficient as the trace
of a weighted colouring matrix. 
We use the same principle here.

For a proper $S$-colouring $\varphi$ of $G$, a {\em state} is a triple 
$$
\alpha_i(\varphi)
=\bigl(\varphi(i), \varphi({i+1}), \varphi({i+2})\bigr).
$$
As $\{i,i+2\}$ is an edge, each state $(a_0,a_1,a_2)$ satisfies $a_0 \ne a_2$.

Let
$$
\Omega=\{(a_0,a_1,a_2)\in S^3:a_0\ne a_2\}.
$$
Let $D$ be the arc-weighted directed graph with vertex set $\Omega$, in which $(\alpha, \beta)$ with $ \alpha = (a_0,a_1,a_2)$ and 
$  \beta =(b_0,b_1,b_2)$ is an arc if
  $b_0 = a_1, b_1=a_2 $
and $b_2 \notin \{a_0,a_1\}$. The weight of this arc is 
$$w(\alpha, \beta) = \dfrac{(a_0-b_2)(a_1-b_2)}{\delta(b_2)}.$$

Let $A$ be the adjacency matrix of $D$, with $A_{\alpha, \beta} = w(\alpha, \beta)$ if $(\alpha, \beta)$ is an arc, and $A_{\alpha, \beta} = 0$ otherwise.

We abbreviate $(a_0,a_1,a_2)$ to the three-symbol string $a_0a_1a_2$
and use the following order for both the rows and columns of $A$:
$$
\begin{gathered}
001,002,011,012,021,022,100,102,110,\\
112,120,122,200,201,210,211,220,221.
\end{gathered}
$$
For example,   $(001, 011)$ and $(001,012)$  are arcs of $D$ and 
$$
A_{001,011}=\frac{(0-1)^2}{\delta(1)}=-1,
\qquad
A_{001,012}=\frac{(0-2)^2}{\delta(2)}=2.
$$
Table~\ref{tab:matrix} lists nonzero entries of $A$.

\begin{table}[htbp]
\centering
\small
\caption{The matrix $A$. An entry $\beta:w$ in row $\alpha$ means
$A_{\alpha,\beta}=w$. All unlisted entries are zero.}
\label{tab:matrix}
\begin{tabular}{c r@{\,}c@{\,}l @{\qquad} c r@{\,}c@{\,}l}
\toprule
Row & \multicolumn{3}{c}{Column: entry}
    & Row & \multicolumn{3}{c}{Column: entry} \\
\midrule
$001$ & $011$ & $:$ & $-1$  & $112$ & $120$ & $:$ & $1/2$ \\
      & $012$ & $:$ & $2$   &       & $122$ & $:$ & $1/2$ \\
\addlinespace[3pt]
$002$ & $021$ & $:$ & $-1$  & $120$ & $200$ & $:$ & $1$   \\
      & $022$ & $:$ & $2$   &       &       &     &       \\
\addlinespace[3pt]
$011$ & $112$ & $:$ & $1$   & $122$ & $220$ & $:$ & $1$   \\
\addlinespace[3pt]
$012$ & $122$ & $:$ & $1$   & $200$ & $001$ & $:$ & $1$   \\
\addlinespace[3pt]
$021$ & $211$ & $:$ & $1$   & $201$ & $011$ & $:$ & $1$   \\
\addlinespace[3pt]
$022$ & $221$ & $:$ & $1$   & $210$ & $100$ & $:$ & $1$   \\
\addlinespace[3pt]
$100$ & $002$ & $:$ & $1$   & $211$ & $110$ & $:$ & $1$   \\
\addlinespace[3pt]
$102$ & $022$ & $:$ & $1$   & $220$ & $200$ & $:$ & $2$   \\
      &       &     &       &       & $201$ & $:$ & $-1$  \\
\addlinespace[3pt]
$110$ & $100$ & $:$ & $1/2$ & $221$ & $210$ & $:$ & $2$   \\
      & $102$ & $:$ & $1/2$ &       & $211$ & $:$ & $-1$  \\
\bottomrule
\end{tabular}
\end{table}

Each proper $S$-colouring $\varphi$ of $G$ corresponds to a closed walk   $W_{\varphi}= (\gamma_0,\ldots,\gamma_{n-1})$ of length $n$ in $D$, where $\gamma_i=(\varphi(i), \varphi(i+1), \varphi(i+2))$. Conversely, each closed walk $W$ in $D$ of length $n$ equals $W_{\varphi}$ for a proper $S$-colouring $\varphi$ of $G$.  For a closed walk  $W_{\varphi}= (\gamma_0,\ldots,\gamma_{n-1})$ of length $n$ in $D$, 

\begin{equation}\label{eq:weightfactorization}
\begin{aligned}
A_{\gamma_0,\gamma_1}\cdots A_{\gamma_{n-1},\gamma_0}&= 
\prod_{i=0}^{n-1}
\frac{(\varphi(i)-\varphi(i+3))
      (\varphi(i+1)-\varphi(i+3))}
     {\delta(\varphi(i+3))}\\[4pt]
&=
\prod_{j=0}^{n-1}
\frac{(\varphi(j-2)-\varphi(j))
      (\varphi(j-3)-\varphi(j))}
     {\delta(\varphi(j))}\\[4pt]
&=\frac{P_G(\varphi)}{D(\varphi)}.
\end{aligned}
\end{equation}

If $(\gamma_0,\ldots,\gamma_{n-1})$ is not a closed walk, then $A_{\gamma_0,\gamma_1}\cdots A_{\gamma_{n-1},\gamma_0}=0$. Therefore,

\begin{equation}
   \label{eq:traceeq} 
   \begin{aligned}
\left[\prod_{i=0}^{n-1}x_i^2\right]P_G(\mathbf x)
&=\sum_{\varphi:V(G)\to S}
  \frac{P_G(\varphi)}{D(\varphi)}\\[4pt]
&=\sum_{\varphi\text{ proper}}
  \frac{P_G(\varphi)}{D(\varphi)}\\[4pt]
&=\sum_{\gamma_0,\ldots,\gamma_{n-1}\in\Omega}
  A_{\gamma_0,\gamma_1}\cdots A_{\gamma_{n-1},\gamma_0}\\[4pt]
&=\operatorname{tr}(A^n).
\end{aligned}
\end{equation}

\begin{lemma}\label{lem:recurrence}
The matrix $A$ satisfies
\begin{equation}\label{eq:annihilator}
A^9-3A^4+A^3+I_{18}=0.
\end{equation}
For $j\ge0$, put $u_j=\operatorname{tr}(A^j)/6$.
Then all $u_j$ are integers,
\begin{equation}\label{eq:urecurrence}
u_j=3u_{j-5}-u_{j-6}-u_{j-9}\qquad(j\ge9),
\end{equation}
and $(u_0,\ldots,u_8)=(3,0,0,0,0,5,-2,0,0)$.
\end{lemma}

\begin{proof}
Exact rational computation of the powers of the matrix in
Table~\ref{tab:matrix} gives~\eqref{eq:annihilator} and
$$
(\operatorname{tr}(A^j))_{j=0}^{8}
=(18,0,0,0,0,30,-12,0,0),
$$
where $A^0=I_{18}$.
Multiplying~\eqref{eq:annihilator} by $A^{j-9}$ and taking traces
gives~\eqref{eq:urecurrence} after division by $6$.
The initial values are integers, so the recurrence proves
integrality for all $j$. The verification code is given in Appendix A.
\end{proof}

\begin{lemma}\label{lem:nonzero}
For every integer $m\ge0$, $u_{6m}\not\equiv0\pmod9$.
\end{lemma}

\begin{proof}
Reducing~\eqref{eq:urecurrence} modulo $3$ gives
$u_j\equiv-u_{j-6}-u_{j-9}\pmod3$.
The initial terms $u_1,u_4,u_7$ are zero.
Induction on indices congruent to $1$ modulo $3$ gives
$u_{3k+1}\equiv0\pmod3$ for every $k\ge0$.

Put $w_k=u_{3k}$. Taking $j=3k+9$ in~\eqref{eq:urecurrence},
we obtain
\begin{equation}\label{eq:vrecurrence}
w_{k+3}=3u_{3k+4}-w_{k+1}-w_k
\equiv-w_{k+1}-w_k\pmod9,
\end{equation}
since $u_{3k+4}\equiv0\pmod3$.
The initial triple is $(w_0,w_1,w_2)\equiv(3,0,7)\pmod9$.
Successive applications of~\eqref{eq:vrecurrence} give
\begin{center}\small
\begin{tabular}{c|rrrrrrrrrrrr}
$k$&0&1&2&3&4&5&6&7&8&9&10&11\\\hline
$w_k\bmod9$&3&0&7&6&2&5&1&2&3&6&4&0\\[4pt]
$k$&12&13&14&15&16&17&18&19&20&21&22&23\\\hline
$w_k\bmod9$&8&5&1&5&3&3&1&3&5&5&1&8
\end{tabular}
\end{center}
The next three residues are $(w_{24},w_{25},w_{26})\equiv(3,0,7)\pmod9$.
By~\eqref{eq:vrecurrence}, three consecutive residues determine
all subsequent residues.
Hence $w_{k+24}\equiv w_k\pmod9$ for every $k\ge0$.
The twelve even positions of the displayed block have residues
$3,7,2,1,3,4,8,1,3,1,5,1$, none of which is zero.
Since $u_{6m}=w_{2m}$, the result follows.
\end{proof}

\begin{proof}[Proof of Theorem~\ref{thm:c6m}]
For $m\ge2$, equation \eqref{eq:traceeq} and Lemma~\ref{lem:nonzero} give
$$
\left[\prod_{i=0}^{6m-1}x_i^2\right]
P_{C_{6m}(2,3)}(\mathbf x)=6u_{6m}\ne0.
$$
Thus $\AT(C_{6m}(2,3))\le3$.
Since $C_{6m}(2,3)$ is $4$-regular, its Alon--Tarsi number is
at least $3$, and the result follows.
\end{proof}

\subsection{The proof of Theorem~\ref{thm:c13}}\label{subsec:cn13}

Our proof is motivated by Tarsi's work on the relation between graph polynomials and the number of proper $3$-colourings \cite{Tarsi1999}. Let $\omega$ be a primitive cube root of unity
and put $S=\{1,\omega,\omega^2\}$. In this subsection, we apply Lemma \ref{lem:coefficient} over $\mathbb C$.

\begin{lemma}\label{lem:roots-four-cycle}
If $x,y$ are distinct elements of $S$, then 

\begin{equation}
    \label{eq:square relation}
    (x-y)^2=-3xy.
\end{equation}
If $a,b,c,d\in S$ satisfy
$a\ne b$, $b\ne c$, $c\ne d$, and $d\ne a$, then
\begin{equation}
    \label{eq:triple relation}
    (a-b)(b-c)(c-d)=3bc(a-d).
\end{equation}
\end{lemma}

\begin{proof}
For distinct $x,y\in S$, as
$0=(x^3-y^3) = (x-y)(x^2+xy+y^2)$, we have $x^2+xy+y^2=0$, and hence $(x-y)^2=-3xy$.
Assume $a,b,c,d\in S$ satisfy
$a\ne b$, $b\ne c$, $c\ne d$, and $d\ne a$.
Since $S$ has three elements,
either $a=c$ or $b=d$.

If $a=c$, then
\[
(a-b)(b-c)(c-d)
=-(a-b)^2(a-d)
=3ab(a-d)
=3bc(a-d).
\]
If $b=d$, then
\[
(a-b)(b-c)(c-d)
=-(a-b)(b-c)^2
=3bc(a-b)
=3bc(a-d).
\]
\end{proof}

\begin{proof}[Proof of Theorem \ref{thm:c13}]
Let $n\ge7$, and let
$\varphi:V(C_n(1,3))\to S$ be a proper colouring.

Let $z_i=\varphi(i)$, $Z=\prod_{i=0}^{n-1}z_i$, and
$L=\prod_{i=0}^{n-1}(z_i-z_{i+1})$. Then 
\[
P_{C_n(1,3)}(\varphi)
=
L\prod_{i=0}^{n-1}(z_i-z_{i+3}).
\]

By equation \eqref{eq:triple relation}, we have
\[
(z_i-z_{i+1})(z_{i+1}-z_{i+2})(z_{i+2}-z_{i+3})
=
3z_{i+1}z_{i+2}(z_i-z_{i+3}).
\]
Multiplying over $i\in\mathbb Z_n$, we obtain
\[
L^3
=
3^nZ^2\prod_{i=0}^{n-1}(z_i-z_{i+3}).
\]
By equation \eqref{eq:square relation}, we have $(z_i-z_{i+1})^2=-3z_iz_{i+1}$, so
$L^2=(-3)^nZ^2$. Therefore
\[
P_{C_n(1,3)}(\varphi)
=
L\prod_{i=0}^{n-1}(z_i-z_{i+3})
=
\frac{L^4}{3^nZ^2}
=
3^nZ^2.
\]
For $z\in S$, differentiating
$\prod_{a\in S}(x-a)=x^3-1$ at $x=z$ gives
\[
\prod_{a\in S\setminus\{z\}}(z-a)=3z^2.
\]
Thus  \[
\frac{P_{C_n(1,3)}(\varphi)}
{\displaystyle
 \prod_{i=0}^{n-1}
 \prod_{a\in S\setminus\{\varphi(i)\}}
 (\varphi(i)-a)}
=1.
\]

Applying Lemma \ref{lem:coefficient},
$$
\left[\prod_{i=0}^{n-1}x_i^2\right]
P_{C_n(1,3)}(\mathbf x)
=
\sum_{\varphi:V(C_n(1,3))\to S}
\frac{P_{C_n(1,3)}(\varphi)}
{\displaystyle
 \prod_{i=0}^{n-1}
 \prod_{a\in S\setminus\{\varphi(i)\}}
 (\varphi(i)-a)}=\sum_{\varphi\text{ proper}}
  1=N_3(C_n(1,3)).
$$

If $n$ is even, assign colours alternately by the word
$(01)^{n/2}$. This colouring is proper since both distance lengths
are odd.

If $n$ is odd, write $n=2k+5$. If $k\ge2$, then 
$
(01)^k21202
$
is a proper colouring of $C_n(1,3)$. Hence $\AT(C_n(1,3))=3$ for $n \ge 8$. 
\end{proof}

\bibliographystyle{amsplain}
\bibliography{references}

\providecommand{\bysame}{\leavevmode\hbox to3em{\hrulefill}\thinspace}
\providecommand{\MR}{\relax\ifhmode\unskip\space\fi MR }
\providecommand{\MRhref}[2]{%
  \href{http://www.ams.org/mathscinet-getitem?mr=#1}{#2}
}
\providecommand{\href}[2]{#2}
\begin{thebibliography}{10}

\bibitem{AlonTarsi1992}
Noga Alon and Michael Tarsi, \emph{Colorings and orientations of graphs},
  Combinatorica \textbf{12} (1992), no.~2, 125--134.

\bibitem{CyclesColourings2025}
J{\'a}nos Bar{\'a}t, Zden{\v{e}}k Dvo{\v{r}}{\'a}k, Penny Haxell,
  Franti{\v{s}}ek Kardo{\v{s}}, Borut Lu{\v{z}}ar, Alfr{\'e}d Onderko, Jozef
  Rajn{\'i}k, Roman Sot{\'a}k, and Nikolay Ulyanov, \emph{Open problems of the
  {33rd Workshop on Cycles and Colourings}}, 2025.

\bibitem{CLZ1999}
Gerard~J. Chang, Daphne D.-F. Liu, and Xuding Zhu, \emph{Distance graphs and
  {$T$}-coloring}, J. Combin. Theory Ser. B \textbf{75} (1999), no.~2,
  259--269. \MR{1676892}

\bibitem{FleischnerStiebitz1992}
Herbert Fleischner and Michael Stiebitz, \emph{A solution to a colouring
  problem of {P}. {E}rd{\H{o}}s}, Discrete Mathematics \textbf{101} (1992),
  no.~1--3, 39--48.

\bibitem{Heubergerchromaticofcirculant}
Clemens Heuberger, \emph{On planarity and colorability of circulant graphs},
  Discrete Mathematics \textbf{268} (2003), no.~1--3, 153--169.

\bibitem{HladkyKralSchauz2010}
Jan Hladk{\'y}, Daniel Kr{\'a}l', and Uwe Schauz, \emph{{B}rooks' theorem via
  the {A}lon--{T}arsi theorem}, Discrete Mathematics \textbf{310} (2010),
  no.~23, 3426--3428.

\bibitem{AT-ToroidalGrids2023}
Zhiguo Li, Zeling Shao, Fedor Petrov, and Alexey Gordeev, \emph{The
  {A}lon-{T}arsi number of a toroidal grid}, European J. Combin. \textbf{111}
  (2023), Paper No. 103697, 7, With a preface by Xuding Zhu. \MR{4602837}

\bibitem{LiYeShao2023}
Zhiguo Li, Qing Ye, and Zeling Shao, \emph{The {A}lon--{T}arsi number of two
  kinds of planar graphs}, Electronic Journal of Graph Theory and Applications
  \textbf{11} (2023), no.~1, 155--164.

\bibitem{Lonelyrunner2025}
Guillem Perarnau and Oriol Serra, \emph{The lonely runner conjecture turns 60},
  Comput. Sci. Rev. \textbf{58} (2025), Paper No. 100798, 13. \MR{4943662}

\bibitem{Schauz2010}
Uwe Schauz, \emph{A paintability version of the {C}ombinatorial
  {N}ullstellensatz, and list colorings of $k$-partite $k$-uniform
  hypergraphs}, Electronic Journal of Combinatorics \textbf{17} (2010), no.~1,
  R176.

\bibitem{Tarsi1999}
Michael Tarsi, \emph{The graph polynomial and the number of proper vertex
  coloring}, Annales de l'Institut Fourier \textbf{49} (1999), no.~3,
  1089--1093.

\bibitem{ZhuBalakrishnan2022}
Xuding Zhu and R.~Balakrishnan, \emph{Combinatorial nullstellensatz: With
  applications to graph colouring}, 1st ed., CRC Focus, Chapman and Hall/CRC,
  Boca Raton, 2022.

\end{thebibliography}

\newpage

\appendix

\section{Exact verification of the matrix calculation}

The following Python code uses rational arithmetic.
It constructs the $18\times18$ matrix $A$ directly
from the transition rule and verifies
$A^9-3A^4+A^3+I_{18}=0$, the initial traces, and
the residue table used in the proof.
The recurrence and the argument extending the modular
period to all indices are given in the main text.

\begin{verbatim}
from fractions import Fraction as F
from itertools import product

S = range(3)
O = [a for a in product(S, repeat=3) if a[0] != a[2]]
idx = {a: i for i, a in enumerate(O)}
A = [[F(0) for _ in O] for _ in O]

for i, (a, b, c) in enumerate(O):
    for d in S:
        if d not in (a, b):
            A[i][idx[(b, c, d)]] = F(
                (a-d)*(b-d), (2, -1, 2)[d])

def mul(X, Y):
    return [
        [sum(X[i][k]*Y[k][j] for k in range(18))
         for j in range(18)]
        for i in range(18)
    ]

P = [[[F(i == j) for j in range(18)]
      for i in range(18)]]
for k in range(9):
    P.append(mul(P[-1], A))

assert all(
    P[9][i][j] - 3*P[4][i][j]
    + P[3][i][j] + P[0][i][j] == 0
    for i in range(18) for j in range(18)
)

t = [sum(X[i][i] for i in range(18)) for X in P[:9]]
assert t == [18, 0, 0, 0, 0, 30, -12, 0, 0]

v = [3, 0, 7]
for k in range(24):
    v.append((-v[k+1] - v[k]) % 9)

assert v[:24] == [
    3, 0, 7, 6, 2, 5, 1, 2, 3, 6, 4, 0,
    8, 5, 1, 5, 3, 3, 1, 3, 5, 5, 1, 8
]
assert v[24:27] == v[:3]
assert all(v[k] != 0 for k in range(0, 24, 2))

print("All checks passed.")
\end{verbatim}

\end{document}